\documentclass[11pt]{amsart}
\usepackage{xcolor,mathtools}

\include{package}
\definecolor{grn}{rgb}{0,0.6,0}
\definecolor{mrn}{rgb}{0.3,0,0}
\definecolor{blue}{rgb}{0,0,0.7}
\definecolor{Mygray}{rgb}{0.75,0.75,0.75}
\definecolor{auburn}{rgb}{0.43, 0.21, 0.1}
\definecolor{britishracinggreen}{rgb}{0.0, 0.26, 0.15}
\definecolor{taupe}{rgb}{0.28, 0.24, 0.2}
\usepackage{amsmath,amssymb}
\newtheorem{theorem}{Theorem}[section]

\newtheorem{propn}{Proposition}[section]

\newtheorem{defn}{Definition}[section]
\newtheorem{quest}{Question}[section]
\newtheorem{rmk}{Remark}[section]

\newtheorem{conj}{Conjecture}[section]

\usepackage{latexsym,amssymb}
\usepackage{amssymb}
\usepackage{graphicx}
\usepackage{ textcomp }
\usepackage[left=2cm,right=2cm,top=2cm,bottom=2cm]{geometry}

\begin{document}
\baselineskip=14.5pt
\title[A brief survey of recent results on P\'{o}lya groups]{A brief survey of recent results on P\'{o}lya groups}

\author{Jaitra Chattopadhyay and Anupam Saikia}
\address[Jaitra Chattopadhyay and Anupam Saikia]{Department of Mathematics, Indian Institute of Technology Guwahati, Guwahati - 781039, Assam, India}
\email[Jaitra Chattopadhyay]{jaitra@iitg.ac.in; chat.jaitra@gmail.com}

\email[Anupam Saikia]{a.saikia@iitg.ac.in}

\begin{abstract}
The P\'{o}lya group of an algebraic number field is a particular subgroup of the ideal class group. This article provides an overview of recent results on P\'{o}lya groups of number fields, their connection with the ring of integer-valued polynomials and touches upon some results on number fields having large P\'{o}lya groups. For the sake of completeness, we have included the proof of Zantema's theorem which laid the foundation to determine the P\'{o}lya groups of many finite Galois extensions over $\mathbb{Q}$. Towards the end of the article, we provide an elementary proof of a weaker version of a recent result of Cherubini et al.
\end{abstract}

\renewcommand{\thefootnote}{}



\footnote{\emph{We confirm that all the data are included in the article.}}

\renewcommand{\thefootnote}{\arabic{footnote}}
\setcounter{footnote}{0}

\maketitle

\section{Introduction}

For an algebraic number field $K$ of degree $d \geq 1$ with ring of integers $\mathcal{O}_{K}$ and discriminant $d_{K}$, let $Cl_{K}$ be the ideal class group of $K$ and let $h_{K}$ be the class number. The class group $Cl_{K}$ is an important object of investigation to algebraic number theorists, particularly because of the arithmetic information about $K$ it carries. Even though the class number $h_{K}$ is finite for all number field $K$, it is quite a difficult task to determine number fields with a given class number. In this regard, Gauss asked for a classification of imaginary quadratic fields of class number $1$, which was completely characterized in the later part of the last century by the seminal works of Baker, Heegner and Stark (cf. \cite{baker}, \cite{stark}). However, the {\it class number $1$ conjecture} of Gauss that predicts the existence of infinitely many real quadratic fields with class number $1$, still remains far from being resolved.

\smallskip

On the other hand, a slightly weaker question of the divisibility of class numbers by a given integer has an affirmative answer. In \cite{nagell} (and also later in \cite{AC}), it was proven that there exist infinitely many imaginary quadratic fields whose class numbers are all divisible by a given integer $g$. Later, in \cite{moto} (also in \cite{berger}), the analogue for real quadratic fields was established. Since then, several properties of class groups and class numbers of number fields of varied degrees have extensively been studied. The interested reader is encouraged to look at \cite{byeon}, \cite{byeon-koh}, \cite{kalyan}, \cite{kishi}, \cite{luca}, \cite{sound} and the references listed therein. Recently, Bhand and Murty \cite{bhand} have written an elaborate article on the class numbers of quadratic fields and it gives a very nice overview of the recent results on this theme.

\smallskip

In the present article, we are interested in a particular subgroup $Po(K)$ of $Cl_{K}$, known as the {\it P\'{o}lya group} of $K$. It stems from the study of a seemingly unrelated object called the ring of {\it integer-valued polynomials} of $K$, which is defined as ${\rm{Int}}(\mathcal{O}_{K}) := \{f(X) \in K[X] : f(\mathcal{O}_{K}) \subseteq \mathcal{O}_{K}\}$. Not only does ${\rm{Int}}(\mathcal{O}_{K})$ form a ring under point-wise addition and multiplication, but also it can be endowed with an $\mathcal{O}_{K}$-module structure by defining $(\alpha,f(X)) \mapsto \alpha f(X)$ for $\alpha \in \mathcal{O}_{K}$ and $f(X) \in {\rm{Int}}(\mathcal{O}_{K})$. In order to get the reader acquainted with the algebraic properties of ${\rm{Int}}(\mathcal{O}_{K})$, let us quickly make some fundamental observations about the case $K = \mathbb{Q}$ closely.

\smallskip

By definition, ${\rm{Int}}(\mathbb{Z})$ consists of polynomials $f(X) \in \mathbb{Q}[X]$ such that $f(t) \in \mathbb{Z}$ for all $t \in \mathbb{Z}$. For example, $\frac{X(X - 1)}{2}$, $\frac{X(X - 1)(X - 2)}{6}$ are elements of ${\rm{Int}}(\mathbb{Z})$. An inductive argument (cf. \cite[Proposition I.1.1.]{cahen-chabert-book}) leads to the fact that any element of ${\rm{Int}}(\mathbb{Z})$ can be uniquely expressed as a $\mathbb{Z}$-linear combination of finitely many $f_{i}$'s, where $f_{i}(X) := \frac{X(X - 1)\cdots (X - i +1)}{i!}$ for $i \geq 1$ and $f_{0}(X) := 1$. In other words, ${\rm{Int}}(\mathbb{Z})$ has a free $\mathbb{Z}$-basis consisting of polynomials of each degree. Such a basis is often referred to as a {\it regular basis}. This property can be used to define {\it P\'{o}lya fields} as follows.

\begin{defn} \cite[Definition II.4.1]{cahen-chabert-book}
A number field $K$ is called a P\'{o}lya field if ${\rm{Int}}(\mathcal{O}_{K})$ admits a regular basis.
\end{defn}

The behaviour of ${\rm{Int}}(\mathcal{O}_{K})$ depends heavily on certain ideal classes of the class group $Cl_{K}$. The following notion of {\it P\'{o}lya groups} facilitates us to bridge a connection.

\begin{defn} (cf. \cite{cahen-chabert-book}, \textsection II.4)
The P\'{o}lya group $Po(K)$ of $K$ is defined to be the subgroup of $Cl_{K}$ generated by the ideal classes $[\Pi_{q}(K)]$, where 
\begin{equation}
\displaystyle\Pi_{q}(K) = \left\{\begin{array}{ll}
\displaystyle\prod_{\substack {\mathfrak{p} \in {\rm{Spec}}(\mathcal{O}_{K})\\ N_{K/\mathbb{Q}} \mathfrak{p} = q}}\mathfrak{p} &\mbox{; if } \mathcal{O}_{K} \mbox{ has primes } \mathfrak{p} \mbox{ of norm } q,\\
\mathcal{O}_{K} &\mbox{; otherwise}.
\end{array}\right.
\end{equation}
\end{defn}

Now that we have introduced both P\'{o}lya fields and P\'{o}lya groups, one might wonder as to how these two are related. It is known that $K$ is a P\'{o}lya field if and only if the P\'{o}lya group $Po(K)$ is trivial. Thus the algebraic structure of the $\mathcal{O}_{K}$-module ${\rm{Int}}(\mathcal{O}_{K})$ is governed, to a large extent, by the structure of $Po(K)$. In other words, the P\'{o}lya group is a measure of failure of a number field from being a P\'{o}lya field. To gather more information about the algebraic properties of the ring of integer-valued polynomials, we direct the reader to look at the informative article \cite{cahen-chabert-monthly}.

\smallskip

After the introduction of P\'{o}lya fields, it is natural to look for number fields that are P\'{o}lya fields. In other words, we look for number fields with trivial P\'{o}lya groups. Indeed, any field of class number $1$ is a P\'{o}lya field as $Po(K) \subseteq Cl_{K}$. Since it is unknown whether there exist infinitely many number fields $K$ with $h_{K} = 1$, the mere inclusion $Po(K) \subseteq Cl_{K}$ of sets is insufficient to conclude the infinitude of number fields that are P\'{o}lya. It is desirable to have an infinite family of P\'{o}lya number fields and towards this direction, it is known due to the work of Zantema \cite[Proposition 2.6]{zantema} that cyclotomic fields are P\'{o}lya fields. Along a similar line, Leriche \cite[Corollary 3.2]{ler-em} proved that the Hilbert class field $H_{K}$ of any number field $K$ is a P\'{o}lya field. Leriche's motivation to prove this originated from the work of Golod and Shafarevich \cite{golod} where they proved that there exist number fields that cannot be embedded inside a number field of class number $1$. Leriche's theorem shows that every number field can be embedded inside a P\'{o}lya field, viz. its Hilbert class field.

\section{Some families of P\'{o}lya fields of small degrees}

Due to the results of Leriche and Zantema, there are plenty of number fields that are P\'{o}lya. In this section, we record some quadratic, cubic and bi-quadratic P\'{o}lya fields. But before we do so, let us recall the following result of Hilbert which connects the order of the P\'{o}lya group of a quadratic field $K$ to the number of ramified  primes in $K/\mathbb{Q}$ and the sign of the fundamental unit. 

\begin{propn} (cf. \cite{hilbert}, \cite[Proposition 1.4]{cahen-chabert-book})\label{hilb}
Let $K = \mathbb{Q}(\sqrt{d})$ be a quadratic field and let $r_{K}$ denote the number of ramified primes in $K/\mathbb{Q}$. Then 
\begin{equation}
|Po(K)| = \left\{\begin{array}{ll}
2^{r_{K} - 2} &\mbox{; if } d > 0 \mbox{ and } N_{K/\mathbb{Q}}(\mathcal{O}_{K}^{*}) = \{1\},\\
2^{r_{K} - 1} &\mbox{; otherwise}.
\end{array}\right.
\end{equation}
\end{propn}

It is clear from Proposition \ref{hilb} that if $K = \mathbb{Q}(\sqrt{d})$ is a P\'{o}lya field, then $d$ cannot have more than two prime factors. Zantema exploited this and classified the quadratic P\'{o}lya fields as follows.

\begin{propn} \cite{zantema}\label{quadratic-poly-classification}
The quadratic field $K = \mathbb{Q}(\sqrt{d})$ P\'{o}lya if and only if one of the following conditions holds.
\begin{enumerate}
\item $d = p$, where $p > 0$ is an odd prime number,

\item $d = 2p$, where $p$ is a prime number and either $p \equiv 3 \pmod {4}$ or $p \equiv 1 \pmod {4}$ and the fundamental unit of $K$ has norm $1$,

\item $d = pq$, where $p$ and $q$ are prime numbers with either $p \equiv q \equiv 3 \pmod {4}$ or $p \equiv q \equiv 1 \pmod {4}$ and the fundamental unit of $K$ has norm $1$,

\item $d = -1, -2$ or $2$,

\item $d = -p$, where $p \equiv 3 \pmod {4}$ is a prime number.
\end{enumerate}
\end{propn}

In \cite{leriche}, Leriche characterized cyclic cubic P\'{o}lya fields. She made use of \cite[Theorem 1]{lem} to deduce that for a cyclic cubic field $K$, the equality $|Po(K)| = 3^{r_{K} - 1}$ holds, where $r_{K}$ stands for the number of ramified primes in $K/\mathbb{Q}$. From this, she proved the following proposition.

\begin{propn} \cite[Proposition 3.2]{leriche}
A cyclic cubic field $K = \mathbb{Q}(\alpha)$ is P\'{o}lya if and only if the minimal polynomial of $\alpha$ over $\mathbb{Q}$ is either $X^{3} - 3X + 1$ or of the form $X^{3} - 3pX - pu$, where $p = \frac{u^{2} + 27w^{2}}{4}$ is a prime number, $w \geq 1$ is an integer and $u$ is an integer with $u \equiv 2 \pmod {3}$.
\end{propn}

The characterization of P\'{o}lya fields, that are Galois extensions of degree $4$ over $\mathbb{Q}$, falls essentially into two categories, viz. cyclic and bi-quadratic fields. Leriche \cite{leriche} classified the cyclic quartic P\'{o}lya fields as follows.

\begin{theorem} \cite[Theorem 4.4]{leriche}
Let $K = \mathbb{Q}(\sqrt{A(D + B\sqrt{D})})$ be a cyclic quartic field, where $A$ is an odd, square-free integer, $D = B^{2} + C^{2}$ is square-free with $B > 0$, $C > 0$ and $\gcd(A,D) = 1$. Then $K$ is a P\'{o}lya field if and only if one of the following conditions holds.

\begin{enumerate}
\item $K = \mathbb{Q}\left(\sqrt{2 + \sqrt{2}}\right)$ or $K = \mathbb{Q}\left(\sqrt{-(2 + \sqrt{2}})\right)$,

\item $K = \mathbb{Q}\left(\sqrt{q(2 + \sqrt{2}})\right)$, where $q$ is an odd prime and $N_{K/\mathbb{Q}}(\mathcal{O}_{K}^{*}) = \{1\}$,

\item $K = \mathbb{Q}\left(\sqrt{p + B\sqrt{p}}\right)$, where $p \equiv 1 \pmod {4}$ is prime, $B \equiv 0 \pmod {4}$ and $p = B^{2} + C^{2}$,

\item $K = \mathbb{Q}\left(\sqrt{-(p + B\sqrt{p}})\right)$, where $p \equiv 1 \pmod {4}$ is prime, $B \equiv 2 \pmod {4}$ and $p = B^{2} + C^{2}$,

\item $K = \mathbb{Q}\left(\sqrt{p + B\sqrt{p}}\right)$, where $p \equiv 1 \pmod {4}$, $B \not\equiv 0 \pmod {4}$, $p = B^{2} + C^{2}$ and $N_{K/\mathbb{Q}}(\mathcal{O}_{K}^{*}) = \{1\}$,

\item $K = \mathbb{Q}\left(\sqrt{q(p + B\sqrt{p})}\right)$, where $p \equiv 1 \pmod {4}$ is a prime, $p = B^{2} + C^{2}$, $q + B \equiv 1 \pmod {4}$ and $N_{K/\mathbb{Q}}(\mathcal{O}_{K}^{*}) = \{1\}$.
\end{enumerate}
\end{theorem} 

Before we state Leriche's classification of bi-quadratic P\'{o}lya fields under the hypotheses that two of its quadratic subfields are P\'{o}lya, we record the following result of Zantema that provides some sufficient conditions for the compositum of two Galois P\'{o}lya fields to be again a P\'{o}lya field.

\begin{theorem} \cite[Theorem 3.4]{zantema}
Let $K_{1}$ and $K_{2}$ be number fields that are Galois over $\mathbb{Q}$. Let $L = K_{1}K_{2}$ be their compositum and let $K = K_{1} \cap K_{2}$. For a rational prime $p$, let $e_{1}(p)$ (respectively, $e_{2}(p)$) be the ramification index of $p$ in $K_{1}/\mathbb{Q}$ (respectively, in $K_{2}/\mathbb{Q}$). Then the following hold.

\begin{enumerate}
\item If both $K_{1}$ and $K_{2}$ are P\'{o}lya fields and $\gcd(e_{1}(p),e_{2}(p)) = 1$ for all rational primes $p$, then $L$ is a P\'{o}lya field.

\item Assume that either $[K_{1} : K], [K_{2} : K]$ and $[K : \mathbb{Q}]$ are pairwise relatively prime, or $K$ itself is a P\'{o}lya field. If $L$ is a P\'{o}lya field, then both $K_{1}$ and $K_{2}$ are P\'{o}lya fields.
\end{enumerate}
\end{theorem}

Now, we state Leriche's result about bi-quadratic P\'{o}lya fields in the next theorem.

\begin{theorem} \cite[Theorem 5.1]{leriche}\label{bi-quadratic-polya-classification}
Let $\mathbb{Q}(\sqrt{m})$ and $\mathbb{Q}(\sqrt{n})$ be distinct quadratic P\'{o}lya fields. Then the bi-quadratic field $\mathbb{Q}(\sqrt{m}, \sqrt{n})$ is P\'{o}lya except for the following cases.
\begin{enumerate}
\item $\mathbb{Q}(\sqrt{-2}, \sqrt{p})$ is not a P\'{o}lya field, where $p$ is a prime with $p \equiv 3 \pmod {4}$.

\item $\mathbb{Q}(\sqrt{-1}, \sqrt{2q})$ is not a P\'{o}lya field, where $q$ is an odd prime number.

\end{enumerate}

Moreover, if for two prime numbers $p$ and $q$, the field $\mathbb{Q}(\sqrt{p}, \sqrt{2q})$ is P\'{o}lya, then 
\begin{enumerate}
\item either $p \equiv -1 \pmod {8}$ and $q \equiv \pm {1} \pmod {8}$ or

\item $p \equiv 3 \pmod {8}$ and $q \equiv 1, 3 \pmod {8}$.
\end{enumerate}
\end{theorem}

Using \cite[Theorem 1]{lem}, we can conclude that if $K$ is a cyclic, quintic extension, then $|Po(K)| = 5^{r_{K} - 1}$, where $r_{K}$ stands for the number of ramified primes in $K/\mathbb{Q}$. Therefore, a cyclic, quintic field is a P\'{o}lya field if and only if there is exactly one ramified prime in $K/\mathbb{Q}$. Leriche also characterized sextic P\'{o}lya fields that are Galois over $\mathbb{Q}$. The cyclic sextic P\'{o}lya fields are quite straightforward to classify and are given by those which are obtained as a compositum of a quadratic P\'{o}lya field and a cyclic cubic P\'{o}lya field (cf. \cite[Proposition 6.1]{leriche}). Leriche classified sextic Polya fields that contain a pure cubic field as a subfield. In other words, she characterized P\'{o}lya fields of the form $K = \mathbb{Q}(\omega,\sqrt[3]{m})$, where $\omega$ is a complex cube root of unity.

\begin{theorem} \cite[Theorem 6.2]{leriche}
Let $m = ab^{2} \geq 2$ be a cube-free integer and let $K = \mathbb{Q}(\omega,\sqrt[3]{m})$, $K_{1} = \mathbb{Q}(\omega)$ and $K_{2} = \mathbb{Q}(\sqrt[3]{m})$. Then $K$ is a P\'{o}lya field if and only if the following hold.

\begin{enumerate}
\item If $a^{2} \not\equiv b^{2} \pmod {9}$, then for each prime number $p$ with $p \mid 3m$, there exists an element $\alpha \in K_{2}$ such that $N_{K_{2}/\mathbb{Q}}(\alpha) = \pm p$.

\item If $a^{2} \equiv b^{2} \pmod {9}$, then for each prime number $p$ with $p \mid m$, there exists an element $\alpha \in K_{2}$ such that $N_{K_{2}/\mathbb{Q}}(\alpha) = \pm p$.
\end{enumerate}
\end{theorem}

\section{Zantema's result on P\'{o}lya fields and some applications}

The theorems quoted in the previous section are concerned about some Galois extensions of $\mathbb{Q}$ being P\'{o}lya and in some cases they were derived using a result proved by Zantema in \cite{zantema}. Zantema considered a finite Galois extension $K$ over $\mathbb{Q}$ and studied the connection between $Po(K)$ and the set of ramified primes in $K/\mathbb{Q}$. In this section, we state his result and briefly outline a proof of the same.

\smallskip

For a finite Galois extension $K/\mathbb{Q}$, let $G := {\rm{Gal}}(K/\mathbb{Q})$. Then $G$ acts on the unit group $\mathcal{O}_{K}^{*}$ by the rule $(\sigma,\alpha) \mapsto \sigma(\alpha)$ for $\sigma \in G$ and $\alpha \in \mathcal{O}_{K}^{*}$. Thus $\mathcal{O}_{K}^{*}$ is endowed with a $G$-module structure. Zantema established a connection between the cohomology group $H^{1}(G,\mathcal{O}_{K}^{*})$ and the P\'{o}lya group $Po(K)$ via the ramified primes as follows.

\begin{theorem} \cite[Page 163]{zantema}\label{zmt}
Let $K/\mathbb{Q}$ be a finite Galois extension with Galois group $G := {\rm{Gal}}(K/\mathbb{Q})$ and let $p_{1},\ldots,p_{t}$ be all the ramified primes in $K/\mathbb{Q}$. For each $i \in \{1,\ldots,t\}$, let $e_{i}$ be the ramification index of $p_{i}$. Then there is an exact sequence as follows.
\begin{equation}\label{exact-equn}
0 \to H^{1}(G,\mathcal{O}_{K}^{*}) \to \displaystyle\bigoplus_{i = 1}^{t}\mathbb{Z}/e_{i}\mathbb{Z} \to Po(K) \to 0.
\end{equation}
\end{theorem}
\begin{proof}
For a rational prime $p$, let $e_{p}$ and $f_{p}$ denote the ramification index and the residual degree, respectively. Let $$p\mathcal{O}_{K} = (\mathfrak{p}_{1}\cdots \mathfrak{p}_{g})^{e_{p}}$$ be the decomposition of the rational prime $p$ into the product of prime ideals in $\mathcal{O}_{K}$. Since $G$ acts transitively on $\{\mathfrak{p}_{1},\ldots ,\mathfrak{p}_{g}\}$, we conclude that $\displaystyle\Pi_{p^{f_{p}}} = \displaystyle\prod_{\substack{\mathfrak{m} \in Spec (\mathcal{O}_{K}) \\ N(\mathfrak{m}) = p^{f_{p}}}}\mathfrak{m}$ is invariant under the action of $G$.

\medskip

Let $I(K)$ be the group of all fractional ideals of $K$ and let $I(K)^{G}$ denote the subgroup of all fractional ideals that are invariant under the action of $G$. Then clearly, $\displaystyle\Pi_{p^{f_{p}}} \in I(K)^{G}$. Also, since the action of $G$ on the prime ideals lying above a rational prime is transitive, we obtain that $\{\displaystyle\Pi_{p^{f_{p}}} : p \mbox{ is a prime in } \mathbb{Z}\}$ generates $I(K)^{G}$.

\medskip

Let $\mathcal{C} = \displaystyle\bigoplus_{i = 1}^{t}\mathbb{Z}/e_{i}\mathbb{Z}$ and we define a group homomorphism $F : I(K)^{G} \to \mathcal{C}$ by setting $$F(\displaystyle\Pi_{p^{f_{p}}}) = (1 \pmod {e_{1}},\ldots ,1 \pmod {e_{t}})$$ for prime numbers $p$ and then we extend it additively to $I(K)^{G}$. Note that, $F$ is surjective because for a given element $(\bar{x_{1}},\ldots ,\bar{x_{t}}) \in \mathcal{C}$, we have $F\left(\displaystyle(\Pi_{p_{1}^{f_{p_{1}}}})^{x_{1}})\cdots (\displaystyle(\Pi_{p_{t}^{f_{p_{t}}}})^{x_{t}}\right) = (\bar{x_{1}},\ldots ,\bar{x_{t}})$, where $p_{i}$'s are the ramified primes in $K/\mathbb{Q}$.

\medskip

Now, we determine ${\rm{ker}}(F)$. For that, let $\left(\displaystyle(\Pi_{p_{1}^{f_{p_{1}}}})^{x_{1}})\cdots (\displaystyle(\Pi_{p_{t}^{f_{p_{t}}}})^{x_{t}}\right)\cdot \displaystyle\prod_{p \neq p_{i}}(\Pi_{p^{f_{p}}})^{x_{p}} \in {\rm{ker}} (F)$. That is, $e_{i} \mid x_{i}$ for all $i \in \{1,\ldots,t\}$. Since $\Pi_{p^{f_{p}}} = p\mathcal{O}_{K}$ for all $p \neq p_{i}$, we conclude that ${\rm{ker}} (F) = \{\alpha\mathcal{O}_{K} : \alpha \in \mathbb{Q}_{>0}\} \simeq \mathbb{Q}_{>0} \simeq \mathbb{Q}^{*}/\{\pm {1}\}$. Therefore, from the exact sequence $$1 \to \mathbb{Q}^{*}/\{\pm {1}\} \to I(K)^{G} \to \mathcal{C} \to 1,$$ we obtain that
\begin{equation}\label{C-iso}
I(K)^{G}/(\mathbb{Q}^{*}/\{\pm {1}\}) \simeq \mathcal{C}.
\end{equation}
Again, if we denote by $\mathcal{P}(K)$ the group of principal fractional ideals of $K$, then from the exact sequence $$1 \to \mathcal{O}_{K}^{*} \to K^{*} \to \mathcal{P}(K) \to 1,$$ we obtain $$1 \to H^{0}(G,\mathcal{O}_{K}^{*}) \to H^{0}(G,K^{*}) \to H^{0}(G,\mathcal{P}(K)) \to H^{1}(G,\mathcal{O}_{K}^{*}) \to H^{1}(G,K^{*}).$$ Using Hilbert's Theorem 90, we obtain $$1 \to \{\pm {1}\} \to \mathbb{Q}^{*} \xrightarrow {f} \mathcal{P}(K)^{G} \xrightarrow {g} H^{1}(G,\mathcal{O}_{K}^{*}) \to 1 \mbox{ is exact }.$$ Consequently, we have 
\begin{equation}\label{pkg-iso}
\mathcal{P}(K)^{G}/(\mathbb{Q}^{*}/\{\pm {1}\}) \simeq H^{1}(G,\mathcal{O}_{K}^{*}).
\end{equation}
Now, the inclusion map $\mathcal{P}(K)^{G} \xhookrightarrow{} I(K)^{G}$ induces the injection $\mathcal{P}(K)^{G}/(\mathbb{Q}^{*}/\{\pm {1}\}) \xhookrightarrow{T} I(K)^{G}/(\mathbb{Q}^{*}/\{\pm {1}\})$. From \eqref{pkg-iso}, we get the injection $T : H^{1}(G,\mathcal{O}_{K}^{*}) \xhookrightarrow{} \mathcal{C}$ and the quotient is $I(K)^{G}/\mathcal{P}(K)^{G} \simeq Po(K)$. This completes the proof of the theorem.
\end{proof}

Theorem \ref{zmt} is very powerful to construct number fields with non-trivial P\'{o}lya groups and of varied degrees. P\'{o}lya groups of bi-quadratic fields have been extensively studied in the recent years. We list some of those results below.

\begin{theorem} \cite[Theorem A - D]{rajaei}\label{to-be-referred-to}
Let $p,q$ and $r$ be three distinct odd prime numbers and let $K = \mathbb{Q}(\sqrt{p},\sqrt{qr})$. Then the following hold true.
\begin{enumerate}
\item If $p \equiv 3 \pmod {4}$, $q \equiv 1 \pmod {4}$, $r \equiv 1 \pmod {8}$, $\left(\frac{q}{r}\right) = -1$ and $\left(\frac{p}{r}\right) = 1$, then $Po(K) \simeq \mathbb{Z}/2\mathbb{Z}$.

\item If $p \equiv q \equiv r \equiv 1 \pmod {4}$, $\left(\frac{p}{r}\right) = 1$, $\left(\frac{q}{r}\right) = -1$ and the fundamental unit of $\mathbb{Q}(\sqrt{pqr})$ has positive norm, then $Po(K) \simeq \mathbb{Z}/2\mathbb{Z}$.

\item If $p \equiv q \equiv 3 \pmod {4}$ and $r \equiv 5 \pmod {8}$, then $Po(K)$ is trivial. That is, $K$ is a P\'{o}lya field.

\item If $p \equiv 3 \pmod {4}$, $q \equiv 1 \pmod {4}$, $r \equiv 5 \pmod {8}$, $\left(\frac{p}{r}\right) = 1$ and $\left(\frac{p}{q}\right) = -1$, then $K$ is a P\'{o}lya field. That is, $Po(K)$ is trivial.
\end{enumerate}
\end{theorem}

\begin{theorem} \cite[Theorem 2.4]{rajaei-jnt}
Let $p$ and $q$ be two prime numbers with $p \equiv q \equiv 3 \pmod {8}$. Then $K = \mathbb{Q}(\sqrt{p},\sqrt{2q})$ is a P\'{o}lya field.
\end{theorem}

\begin{theorem} \cite[Theorem 3.3]{rajaei-jnt}
Let $p$ and $q$ be two prime numbers with $p \equiv 1 \pmod {4}$ and $q \equiv 3 \pmod {4}$. If $\left(\frac{q}{p}\right) = -1$, then $K = \mathbb{Q}(\sqrt{-p},\sqrt{-q})$ is a P\'{o}lya field.
\end{theorem}

Recently, the authors proved the following theorems, extending Theorem \ref{to-be-referred-to}. The motivation to consider these fields essentially comes from an earlier work on {\it Euclidean ideal class} in \cite{self-jnt}. The notion of an Euclidean ideal class, abbreviated as EIC for the sake of brevity, was introduced by Lenstra in \cite{lenstra} to extend the notion of Euclidean domain for number fields. He proved that if a number field $K$ has an EIC, then the class group $Cl_{K}$ is cyclic. He further proved, assuming the validity of the Extended Riemann Hypothesis, that the converse also holds true. To learn about the recent developments in this area, we refer the reader to the references listed in \cite{self-jnt}. 

\begin{theorem} \cite[Theorem 1]{self-arxiv}\label{main-1} 
Let $p,q$ and $r$ be distinct prime numbers with $p \equiv 3 \pmod {4}$ and $q \equiv r \equiv 1 \pmod {8}$. Assume that $\left(\frac{q}{r}\right) = -1$, where $\left(\frac{\cdot}{r}\right)$ stands for the Legendre symbol. Then for the bi-quadratic field $K^{\prime} = \mathbb{Q}(\sqrt{p},\sqrt{qr})$, we have $Po(K^{\prime}) \simeq \mathbb{Z}/2\mathbb{Z}$. In particular, $K^{\prime}$ is not a P\'{o}lya field.
\end{theorem}

\begin{theorem} \cite[Theorem 2]{self-arxiv}\label{NEW-TH}
Let $p,q$ and $r$ be distinct prime numbers with $p \equiv q \equiv 3 \pmod {4}$ and $r \equiv 1 \pmod {8}$. Assume that $\left(\frac{p}{r}\right) = 1$ and $\left(\frac{q}{r}\right) = -1$. Then for the the bi-quadratic field $F = \mathbb{Q}(\sqrt{p},\sqrt{qr})$, we have $Po(F) \simeq \mathbb{Z}/2\mathbb{Z}$. In particular, $F$ is not a P\'{o}lya field.
\end{theorem}

\begin{theorem} \cite[Theorem 3]{self-arxiv}\label{main-2}
Let $p$ and $q$ be distinct prime numbers with $p \equiv q \equiv 1 \pmod {4}$. Assume that $\left(\frac{p}{q}\right) = -1$. Then for the bi-quadratic field $K^{\prime \prime} = \mathbb{Q}(\sqrt{2},\sqrt{pq})$, we have $Po(K^{\prime \prime}) \simeq \mathbb{Z}/2\mathbb{Z}$. In particular, $K^{\prime \prime}$ is not a P\'{o}lya field.
\end{theorem}

It is worthwhile to mention that the bi-quadratic fields listed in Theorem \ref{main-1} - Theorem \ref{main-2} are known to have an EIC whenever class groups of these fields are all isomorphic to $\mathbb{Z}/2\mathbb{Z}$ \cite{self-jnt}. We also remark here that some allied questions have been dealt with in the very recent article by Tougma in \cite{recent-jnt}.

\section{Large P\'{o}lya groups of bi-quadratic and quintic fields}

Theorem \ref{zmt} can be used to construct finite Galois extensions $K$ of $\mathbb{Q}$ with large P\'{o}lya groups. The fundamental idea consists of two parts. First, one needs to construct $K$ with its discriminant $d_{K}$ having many prime factors so that the term $\displaystyle\bigoplus_{i = 1}^{t}\mathbb{Z}/e_{i}\mathbb{Z}$ in the exact sequence \eqref{exact-equn} becomes {\it large}. Then one needs to ensure that the cohomology group $H^{1}(G,\mathcal{O}_{K}^{*})$ is {\it small}. Exploiting these ideas, the authors constructed an infinite set of totally real bi-quadratic fields having large P\'{o}lya groups. We state the result as follows.

\begin{theorem} \cite[Theorem 1]{self-rnt}\label{RNT-thm}
For any integer $n \geq 1$, there exist infinitely many totally real bi-quadratic fields with P\'{o}lya groups being isomorphic to $(\mathbb{Z}/2\mathbb{Z})^{n}$.
\end{theorem}

Since $Po(K)$ is a subgroup of $Cl_{K}$, it immediately follows that the class groups of each of the bi-quadratic fields in Theorem \ref{RNT-thm} contains a copy of $(\mathbb{Z}/2\mathbb{Z})^{n}$. In fact, there exist infinitely many totally real bi-quadratic fields whose ideal class groups have $2$-rank at least $n$ (cf. \cite[Corollary 1]{self-rnt}).

\smallskip

Unlike the case of bi-quadratic fields, not much is known about the largeness of $Po(K)$, where $K$ is a quintic number field. Recently, Mahapatra and Pandey \cite{ppp-paper} addressed this question for {\it Lehmer quintics}. We recall that for an integer $n \in \mathbb{Z}$, the Lehmer quintic polynomial $f_{n}(X) \in \mathbb{Z}[X]$ is defined by

\begin{align*}
f_{n}(X) := X^{5} + n^{2}X^{4} - (2n^{3} + 6n^{2} + 10n + 10)X^{3} + (n^{4} + 5n^{3} + 11n^{2} + 15n + 5)X^{2} \\ (n^{3} + 4n^{2} + 10n + 10)X + 1
\end{align*}
and a field $K_{n}$ is called a Lehmer quintic field if $K_{n} = \mathbb{Q}(\alpha_{n})$, where $\alpha_{n}$ is a root of $f_{n}(X)$. In \cite{ppp-paper}, the following theorem is proved for the P\'{o}lya groups of Lehmer quintic fields.

\begin{theorem} \cite[Theorem 1.4]{ppp-paper}\label{ppp}
Let $K_{n}$ be the sequence of Lehmer quintic fields as defined above. Assume that the integer $m_{n} = n^{4} + 5n^{3} 15n^{2} + 25n + 25$ is cube-free. Then the following hold true.
\begin{enumerate}
\item The P\'{o}lya group $Po(K_{n}) \simeq (\mathbb{Z}/5\mathbb{Z})^{\omega(m_{n}) - 1}$, where $\omega(k)$ counts the number of distinct prime divisors of $k$.

\item The field $K_{n}$ is a P\'{o}lya field if and only if either $m_{n}$ is a prime number or $m_{n} = 25$. 
\end{enumerate}
\end{theorem}

As a corollary to Theorem \ref{ppp}, they proved that the sequence of $5$-ranks of the class groups of the non-P\'{o}lya Lehmer quintic fields is unbounded (cf. \cite[Corollary 1.5]{ppp-paper}).
 
\section{Large P\'{o}lya groups of consecutive quadratic fields in the spirit of Iizuka's conjecture}

There have been quite a lot of interest in the question of simultaneous divisibility or indivisibility of class numbers of quadratic fields. This can be traced back to the work of Scholz \cite{scholz}, widely termed as {\it Scholz's reflection principle} where he provided an inequality of the $3$-ranks of the ideal class groups of the quadratic fields $\mathbb{Q}(\sqrt{d})$ and $\mathbb{Q}(\sqrt{-3d})$. Over the last couple of decades, there have been much progress towards generalizing Scholz's result for pairs of quadratic fields. We encourage the reader to look into \cite{byeon-PAMS}, \cite{ito}, \cite{kom-1} and \cite{kom-2} to learn about the recent developments in this regard.

\smallskip

In 2018, Iizuka \cite{iizuka} proved the existence of infinitely many pairs of imaginary quadratic fields of the form $\{\mathbb{Q}(\sqrt{d}),\mathbb{Q}(\sqrt{d + 1})\}$ with $d \in \mathbb{Z}$ such that the class numbers of all the fields are divisible by $3$. In the same paper, he proposed the following conjecture.

\begin{conj} \cite{iizuka}\label{iizuka-conj}
Let $k \geq 1$ be an integer and let $\ell$ be a prime number. Then there exist infinitely many tuples of quadratic fields, real or imaginary, of the form $\{\mathbb{Q}(\sqrt{d}),\ldots\mathbb{Q}(\sqrt{d + k})\}$ with $d \in \mathbb{Z}$ such that the class numbers of all of them are divisible by $\ell$.
\end{conj}

Iizuka proved that for $k = 1$ and $\ell = 3$, Conjecture \ref{iizuka-conj} holds true for imaginary quadratic fields. Later, Krishnamoorthy and Pasupulati settled the conjecture for $k = 1$ and for any odd prime $\ell$ in \cite{pisi}. Even though the conjecture remains wide open, there are certain recent works that establish the existence of particular types of triples, quadruples and $k$-tuples of quadratic fields whose class numbers satisfy certain divisibility conditions (cf. \cite{self-acta}, \cite{azi-da} and \cite{xie}). 

\smallskip

In \cite{imrn}, Cherubini et al. considered a question concerning the distribution of class numbers of real quadratic fields that goes along the spirit of Conjecture \ref{iizuka-conj}. Following the notations of \cite{imrn}, we call a sequence of quadratic fields $\{K_{i}\}_{i \geq 1}$ {\it consecutive} if each $K_{i}$ is given by $K_{i} = \mathbb{Q}(\sqrt{d + i})$ for some fixed integer $d$. In \cite{imrn}, consecutive real quadratic fields with {\it large} class numbers have been shown to exist. More precisely, the following theorem has been proven.

\begin{theorem} \cite{imrn}\label{imrn-thm}
For a fixed integer $k \geq 1$, there exist $\geq X^{\frac{1}{2} - o(1)}$ integers $d \leq X$ such that $$h_{\mathbb{Q}(\sqrt{d + j})} \gg_{k} \dfrac{\sqrt{d}}{\log d}\log \log d, ~ \mbox{ for all } ~ j = 1,\ldots,k$$ holds.
\end{theorem}

Here we consider a weaker version of Theorem \ref{imrn-thm} and prove the following result, using elementary means. For the sake of notational brevity, we use the standard notation $h(d)$ to denote the class number of the quadratic field $\mathbb{Q}(\sqrt{d})$.

\begin{theorem}\label{main-thm}
Let $k \geq 1$ be an integer and let $M$ be a positive real number. Then there exist $k$ consecutive distinct real quadratic fields $\mathbb{Q}(\sqrt{d + 1}),\ldots,\mathbb{Q}(\sqrt{d + k})$ such that $h(d + i) > M$ for all $i = 1,\ldots,k$.
\end{theorem}
\begin{proof}
We choose an integer $n \geq 1$ such that $2^{n} > M$. Our aim is to produce real quadratic fields whose discriminants have at least $n + 2$ prime factors.

\smallskip

We choose $k$ pairwise disjoint sets of prime numbers $\{p_{1,1},\ldots,p_{1,n + 2}\}$, $\ldots$, $\{p_{k,1},\ldots,p_{k,n + 2}\}$ such that $p_{i,j} > k$ for all $i,j$ with $1 \leq i \leq k$ and $1 \leq j \leq n + 2$. We consider the following system of congruence. 
\begin{equation}\label{eqn-main}
X \equiv -i + p_{i,j} \pmod{p_{i,j}^{2}} ~ \mbox{ for } 1 \leq i \leq k ~ \mbox{ and }1 \leq j \leq n + 2.
\end{equation}
By the Chinese remainder theorem, there exists a unique integer $d$ modulo $\displaystyle\prod_{\substack{1 \leq i \leq k \\ 1 \leq j \leq n + 2}} p_{i,j}^{2}$, satisfying the congruence conditions in \eqref{eqn-main}. Therefore, for each such $i$, we have that $d + i \equiv 0 \pmod {p_{i,j}}$ but $d + i \not\equiv 0 \pmod{p_{i,j}^{2}}$ for all $j \in \{1,\ldots,n + 2\}$. Hence each $d + i$ has at least $n + 2$ odd prime factors that appear with exponent $1$ in the prime factorization. Consequently, by Gauss theory of genera, $h(d + i)$ is divisible by $2^{n}$ for each $i \in \{1,\ldots,k\}$. Since $2^{n} > M$ , we conclude that $h(d + i) > M$ for all $i \in \{1,\ldots,k\}$.

\smallskip

Moreover, since $p_{i,j} > k$ for all $i$ and $j$, we see that $d + i^{\prime} \not\equiv 0 \pmod{p_{i,j}}$ for $i^{\prime} \neq i$ because that would imply $i \equiv i^{\prime} \pmod{p_{i,j}}$, which is impossible because $0 < |i - i^{\prime}| < k < p_{i,j}$. Consequently, the rational primes $p_{i,j}$'s for $1 \leq j \leq n + 2$, are ramified in $\mathbb{Q}(\sqrt{d + i})$ and not in $\mathbb{Q}(\sqrt{d + i^{\prime}})$ for any $i^{\prime} \neq i$. Thus the quadratic fields $\mathbb{Q}(\sqrt{d + 1}),\ldots,\mathbb{Q}(\sqrt{d + k}))$ are all distinct. This completes the proof of the theorem.
\end{proof}

\medskip

\begin{rmk}
We note that, in fact, there are infinitely many $k$-tuples of consecutive real quadratic fields $\{\mathbb{Q}(\sqrt{d + 1}),\ldots,\mathbb{Q}(\sqrt{d + k})\}$ such that $h(d + i) > M$ for all $i \in \{1,\ldots,k\}$. Indeed, by Theorem \ref{main-thm}, we first choose one $k$-tuple $\{\mathbb{Q}(\sqrt{d + 1}),\ldots,\mathbb{Q}(\sqrt{d + k})\}$ of consecutive real quadratic fields with $h(d + i) > M$. Let $h := \displaystyle\max_{1 \leq i \leq k}\{h(d + i)\}$. Then $h > M$ and again applying Theorem \ref{main-thm} for $h$ in place of $M$, we obtain another $k$ consecutive real quadratic fields $\{\mathbb{Q}(\sqrt{d^{\prime} + 1}),\ldots,\mathbb{Q}(\sqrt{d^{\prime} + k})\}$ with class numbers of all of them bigger than $h$. Clearly, $\mathbb{Q}(\sqrt{d + i}) \neq \mathbb{Q}(\sqrt{d^{\prime} + j})$ for any $i$ and $j$ because $h(d + i) \leq h$ for all $i$ and $h(d^{\prime} + j) > h$ for all $j$. We can continue this process to obtain infinitely many such $k$-tuples of consecutive real quadratic fields, all of which have class numbers bigger than $M$.
\end{rmk}

We can even consider the ``multiplicative" version of Theorem \ref{main-thm} and have the following theorem.

\medskip

\begin{theorem}\label{mult}
For any integer $k \geq 1$ and any real number $M > 0$, there exist infinitely many square-free integers $d > 0$ with $h(jd) > M$ for all $j = 1,\ldots,k$.
\end{theorem}
\begin{proof}
Let $t \geq 1$ be an integer such that $2^{t} > M$. Let $p_{1},\ldots,p_{t + 2}$ be prime numbers, all bigger than $k$ and let $d = \displaystyle\prod_{i = 1}^{t + 2}p_{i}$. Then by Gauss theory of genera, we have $2^{t}$ divides each integer $h(d),h(2d),\ldots,h(kd)$. Also, the the quadratic fields $\mathbb{Q}(\sqrt{d}),\ldots,\mathbb{Q}(\sqrt{kd})$ are distinct and $2^{t} > M$, the result follows.
\end{proof}

Clearly, Proposition \ref{hilb} tells us that there exist quadratic fields with large P\'{o}lya groups. In the spirit of Theorem \ref{main-thm}, we prove the following theorem for consecutive real quadratic fields with large P\'{o}lya groups.

\begin{theorem}
Let $k \geq 1$ be an integer and let $M$ be a positive real number. Then there exist quadratic fields, real or imaginary, $K_{i} := \mathbb{Q}(\sqrt{d + i})$ for $i = 1,\ldots,k$ such that $|Po(K_{i})| > M$ for each $i$.
\end{theorem}
\begin{proof}
We furnish the proof for the case of real quadratic fields as the case for imaginary quadratic fields is similar. Following the proof of Theorem \ref{main-thm}, we choose the primes $p_{i,j}$'s such that $p_{i,j} \equiv 3 \pmod {4}$ for all $i \in \{1,\ldots,k\}$ and $j \in \{1,\ldots,n + 2\}$. Then $N_{K/\mathbb{Q}}(\varepsilon_{K_{i}}) = 1$. Therefore, each $K_{i}$ has at least $n + 2$ ramified primes in $K_{i}/\mathbb{Q}$. Thus by Proposition \ref{hilb}, we conclude that $|Po(K_{i})| \geq 2^{n} > M$ for each $i$. 
\end{proof}

Following the same line of argument as used in Theorem \ref{mult}, we also obtain the following theorem.

\begin{theorem}\label{mult-pol}
For any integer $k \geq 1$ and any real number $M > 0$, there exist quadratic fields, real or imaginary, $K_{j} := \mathbb{Q}(\sqrt{jd})$ for $i = \{1,\ldots,k\}$ such that $|Po(K_{j})| > M$ for each $j$.
\end{theorem}

\medskip

\section{Concluding remarks}

It can be observed that the distribution of the class numbers are quite erratic along consecutive integers. It may be an interesting question to investigate the {\it closeness} of the class numbers of consecutive real quadratic fields. More precisely, we ask the following question.

\begin{quest}\label{closeness-question}
What is the value of $\lim\inf |h(d + 1) - h(d)|$ as $d \to \infty$?
\end{quest}

It would be quite fascinating to have infinitely many pairs of consecutive real quadratic fields with equal class numbers. In that case, the answer to Question \ref{closeness-question} would be $0$.

\smallskip

By abuse of notation, we denote the order of $Po(\mathbb{Q}(\sqrt{d}))$ by $Po(d)$ and ask the following question in the light of Question \ref{closeness-question}.

\begin{quest}\label{q2}
What is the value of $\lim\inf |Po(d + 1) - Po(d)|$ as $d \to \infty$?
\end{quest}

If there exist infinitely many pairs of consecutive integers with the same number of prime factors appearing in odd exponents, then by using Proposition \ref{hilb} we would be able to conclude that the answer to Question \ref{q2} is also $0$. For example, the pairs of integers $(21,22)$, $(33,34)$, $(34,35)$, $(38,39)$ and possibly more are admissible choice for such consecutive pairs. In view of this, we ask the following question, an affirmative answer to which will resolve Question \ref{q2}.

\begin{quest}
Let $k \geq 1$ be an integer. Are there infinitely many pairs of consecutive integers such that their square-free parts have exactly $k$ distinct prime divisors?
\end{quest}

\begin{rmk}
For $k = 1$, it is related to Fermat primes. For example, if $2^{2^{n}} + 1$ is a Fermat prime, then $(2^{2^{n}},2^{2^{n}} + 1)$ is such a pair. Since conjecturally it is expected that there are possibly finitely many Fermat primes, therefore examples of such kind is also expected to be finite in  number.
\end{rmk}

%

{\bf Acknowledgements.} We are immensely thankful to Prof. Andrew Granville for his insightful comments on the contents of Section 5 and Section 6. The first author gratefully acknowledges the National Board of Higher Mathematics (NBHM) for the Post-Doctoral Fellowship (Order No. 0204/16(12)/2020/R \& D-II/10925). The second author thanks MATRICS, SERB for the research grant MTR/2020/000467.

\end{document}